\documentclass[final,1p,authoryear]{elsarticle}
\usepackage{amssymb,amsfonts,amsmath,amsthm}
\newcommand{\one}{{\bf 1}}
\newcommand{\odin}[1]{{\mbox {\bf I}}_{\{#1\}}}
\newcommand{\witi}{\widetilde}
\newcommand{\fracd}[2]{\frac {\displaystyle #1}{\displaystyle #2 }}
\newcommand{\nn}{{\mathbb N}}
\newcommand{\rr}{{\mathbb R}}
\newcommand{\zz}{{\mathbb Z}}
\newcommand{\cald}{{\mathcal D}}

\newcommand{\cals}{{\mathcal S}}

\newcommand{\calr}{{\mathcal R}}

\newcommand{\calx}{{\mathcal X}}

\newcommand{\ba}{{\bf a}}
\newcommand{\bB}{{\bf b}}

\newcommand{\be}{{\bf e}}
\newcommand{\bk}{{\bf K}}
\newcommand{\bm}{{\bf m}}
\newcommand{\bq}{{\bf q}}
\newcommand{\bx}{{\bf x}}

\newcommand{\veps}{\varepsilon}
\newcommand{\beq}{\begin{eqnarray*}}
\newcommand{\feq}{\end{eqnarray*}}
\newcommand{\beqn}{\begin{eqnarray}}
\newcommand{\feqn}{\end{eqnarray}}
\newcommand{\as}{\mbox{a.s.}}

\newcommand{\proda}{{\mbox {\Large $\Pi$}}}
\newtheorem{theorem}{Theorem}
\makeatletter \@addtoreset{theorem}{section}\makeatother

\newtheorem{definition}[theorem]{Definition}
\newtheorem{lemma}[theorem]{Lemma}
\newtheorem{assume}[theorem]{Assumption}
\newtheorem*{theorema*}{Theorem~A}
\newtheorem*{theoremb*}{Theorem~B}
\newtheorem*{theorem*}{Theorem}
\newtheorem{proposition}[theorem]{Proposition}

\begin{document}
\begin{frontmatter}
\title{Random linear recursions with dependent coefficients\tnoteref{l1}}
\tnotetext[l1]{Most of this work was done during Summer' 09 Research
Experience for Undergraduate (REU) Program at Iowa State University.
D.H, R.R, and A.S. were partially supported by
NSF Award DMS-0750986. J.S. thanks the Department of Mathematics at Iowa State University
for the hospitality and support during her visit there.}
\author{Arka P. Ghosh\fnref{a1,a2}}
\fntext[a1]{Department of Mathematics, Iowa State University, Ames, IA 50011, USA}
\author{Diana Hay\fnref{a1}}
\fntext[a2]{Department of Statistics, Iowa State University, Ames, IA 50011, USA}
\author{Vivek Hirpara\fnref{a1,a3}}
\fntext[a3]{Department of Economics, Iowa State University, Ames, IA 50011, USA}
\author{Reza Rastegar\fnref{a1}}
\author{Alexander Roitershtein\corref{cor1}\fnref{a1}}
\cortext[cor1]{Corresponding author. E-mail: roiterst@iastate.edu}
\author{Ashley Schulteis\fnref{a4}}
\fntext[a4]{Department of Math, Physics and CS, Wartburg College, Waverly, IA 50677, USA}
\author{Jiyeon Suh\fnref{a5}}
\fntext[a5]{Department of Math, Grand Valley State University, Allendale, MI 49401, USA}
\begin{abstract}
We consider the equation $R_n=Q_n+M_n R_{n-1},$ with random non-i.i.d. coefficients $(Q_n,M_n)_{n\in\zz}\in \rr^2,$
and show that the distribution tails of the stationary solution to this equation
are regularly varying at infinity.
\end{abstract}
\end{frontmatter}
\section{Introduction and statement of results}
\label{intro}
\subsection{Outline}
Let $(Q_n,M_n)_{n \in \zz}$ be $\rr^2$-valued random pairs and
consider the recursion
\beqn
\label{main-def}
R_n=Q_n+M_n R_{n-1},~~~~~n \in \nn, ~R_n
\in \rr.
\feqn
This equation has a wide variety of real world and theoretical applications,
see \citep{embre-goldie,vervaat}. Sufficient (in fact, close to necessary) conditions
for $R_n$ to converge in law, independently of $R_0,$ to $R= Q_0+\sum_{n=1}^{\infty} Q_{-n} \prod_{i=0}^{n-1} M_{-i},$
can be found in \citep{brandt}.
\par
The distribution tails of $R$ were shown to be regularly varying  in \citep{kesten-randeq,goldie,trakai75,grey} provided that
the pairs $(Q_n,M_n)_{n\in\zz}$ form an i.i.d. sequence. In the setup of \citep{kesten-randeq,goldie} the tails
are in fact power tailed. Recall that $f:\rr\to\rr$ is called regularly varying if $f(t)=t^\alpha L(t)$ for some $\alpha\in\rr$
where $L(t)$ is a slowly varying function, that is $L(\lambda t)\sim L(t)$ for all $\lambda >0.$
Here and henceforth $f(t)\sim g(t)$ (as a rule, we omit ``as $t\to\infty$") means $\lim_{t\to\infty} f(t)/g(t)=1.$
\par
The mechanisms leading to regularly varying tails of $R$ are different in
\citep{kesten-randeq, goldie} versus \citep{trakai75,grey}. In the former case, there exists a critical exponent $\alpha$
such that $E(|M_n|^\alpha)=1$ and $E(|Q_n|^\alpha)<\infty.$ Then, $R$ is heavy tailed essentially
because of one atypical fluctuation of $S_n=\sum_{i=0}^{n-1} \log |M_{-i}|,$
which follows from renewal arguments for $S_n$ under exponential tilt. Remarkably,
while no particular tail behavior is assumed on inputs, exact power law tails appear in the output.
In contrast, the critical exponent is not available in the latter case, where the tails of $Q_n$ are assumed to be regular varied.
The second case thus provides an instance of the phenomenon ``regular variation in, regular variation out"
for \eqref{main-def}. This setup is appealing because it enables one to gain insight
into the structure and fine properties of the sequence $(R_n)_{n\in\nn}$, including the asymptotic behavior of its partial sums and
extremes, see for instance \citep{maxima,ppps,ldp-rec,samorachev,langevin4}.
\par
The goal of this paper is to study \eqref{main-def} with non-i.i.d. coefficients.
The extension is desirable in many, especially financial, applications. See \citep{infor, msee}.
We remark that an extension of the main result of \citep{kesten-randeq, goldie} to a
Markov setup has been obtained in \citep{saporta,omar,msee}.
\begin{definition}
\label{inde}
The coefficients $(Q_n,M_n)_{n\in\zz}$ are said to be induced by a sequence of random variables
$(X_n)_{n \in \zz},$ each valued in a countable set $\cald,$ if there exist
independent random variables $(Q_{n,i}, M_{n,i})_{n\in\zz,i\in\cald}\in\rr^2$ such that for a fixed $i
\in\cald,$ $(Q_{n,i},M_{n,i})_{n\in\zz}$ are i.i.d,
\beqn
\label{gigs}
Q_n=\sum_{j\in\cald}Q_{n,j}\odin{X_n=j}=Q_{n,X_n},
M_n=\sum_{j\in\cald}M_{n,j}\odin{X_n=j}=M_{n,X_n},
\feqn
and $(Q_{n,i}, M_{n,i})_{n\in\zz,i\in\cald}$ is independent of $(X_n)_{n \in \zz}.$
\end{definition}
Note that if $(X_n)_{n\in\zz}$ is a finite Markov chain, then
\eqref{gigs} defines a {\em Hidden Markov Model} (HMM), see \citep{technion}
for a survey of HMM and their applications. Heavy tailed HMM
are considered for instance in \citep{hmht}, see also references therein.
\subsection{Markov-dependent coefficients: regular variation in, regular variation out}
First, we will impose the following conditions on the coefficients in \eqref{main-def}.
\begin{assume}
\label{agg}
Assume that the sequence $(Q_{n,i}, M_{n,i})_{n\in\zz,i\in\cald}\in\rr^2$ is induced by
a stationary irreducible Markov chain $(X_n)_{n \in \zz}$ on a countable state space $\cald.$
Furthermore, suppose that there exists a constant $\alpha>0$ s.~t.:
\item[(A1)]
There is a slowly varying function $L(t)$ and two sequences of
constants $(q_i^{(\eta)})_{i\in\cald},$ $\eta\in\{-1,1\},$
such that $\lim_{t\to\infty} \frac{P(Q_{n,i}\cdot \eta>t)}{t^{-\alpha}L(t)}=q_i^{(\eta)}$ uniformly in $i\in\cald.$
Moreover, we have $\sum_{j\in\cald}q^{(1)}_j>0$ and $\sup_{j\in\cald} \max\{q^{(1)}_j,q^{(-1)}_j\}<\infty.$
\item[(A2)] There exists $\beta>\alpha$ such that $\sup_{i\in\cald} E(|M_{0,i}|^\beta)<\infty.$
\item[(A3)] Let $m_i^{(\eta)}=E\bigl(|M_{0,i}|^\alpha \odin{M_{0,i}\cdot \eta >0}\bigr),$
$m_i=m_i^{(-1)}+m_i^{(1)}.$ Then $\sup\limits_{i\in\cald} m_i<1.$
\item[(A4)]  $\lim_{\veps\to 0+}P(M_{1,i}\leq \veps)=P(M_{1,i}\leq 0)$ uniformly in $i\in\cald.$
\end{assume}
Let $B_b = \{\bx=(x_n)_{n \in \cald}: \sup_{n \in \cald}|x_n|<+\infty\}\subset\rr^\cald$
be equipped with the norm $\|\bx\|=\sup_{n \in \cald}|x_n|.$
Let $H$ be transition matrix of the stationary backward chain $X_{-n},$
that is $H(i,j)=P(X_n=j|X_{n+1}=i).$ Define matrices $G_\eta,$ $\eta\in\{-1,1\},$ as
$G_\eta(i,j)=m_i^{(\eta)}H(i,j)$ and set $G(i,j)=m_iH(i,j),$ $i,j\in\cald.$
The spectral radius of $G$ (operator in $B_b$) is less than 1 by (A3).
\par
Under Assumption~\ref{agg}, the tails of $R$ and $Q_0$ have similar structure.
Denote $\bq^{(\eta)}=\bigl(q_i^{(\eta)}\bigr)_{i\in\cald},$ $\bm^{(\eta)}=\bigl(m_i^{(\eta)}\bigr)_{i\in\cald},$
$\eta\in\{-1,1\},$ and $\bq=\bq^{(1)}+\bq^{(-1)}.$
\begin{theorem}
\label{vart}  Let Assumptions \ref{agg} hold and
suppose that $P(M_{0,i}>0)=1$ for all $i\in\cald.$ Then, for all $i \in\cald,$
$P(R>t|X_0=i)~\sim K_i t^{-\alpha}L(t),$
where $\bk=(K_i)_{i\in\cald}\in B_b$ is
defined by $\bk=(I-G)^{-1}\bq^{(1)}.$
\end{theorem}
Theorem~\ref{vart} yields its analog without the restriction $P(M_0>0)=1.$
\begin{theorem}
\label{negs}
Let Assumption~\ref{agg} hold. Then,
$P(R\cdot \eta>t|X_0=i)~\sim K_i^{(\eta)} t^{-\alpha}L(t)$ for $\eta\in\{-1,1\},$
where $\bk^{(\eta)}=\bigl(K_i^{(\eta)}\bigr)_{i\in\cald}\in B_b$ are given by
\beq
\bk^{(\eta)}&=&\frac{1}{2}\Bigl((I-G)^{-1}\bigl(\bq^{(1)}+\bq^{(-1)}\bigr)+\eta(I-G_+ +G_-)^{-1}\bigl(\bq^{(1)}-\bq^{(-1)}\bigr)\Bigr).
\feq
\end{theorem}
Theorem~\ref{vart} and~\ref{negs} are proved in Sections~\ref{proofvart} and~\ref{proofa},
respectively.
\subsection{Kesten's power law for coefficients induced by chains of infinite order}
We next consider coefficients induced by process with infinite memory.
\begin{definition}
\label{typc}
A C-chain is a stationary process $(X_n)_{n \in \zz}$ taking values in a finite set (alphabet)
$\cald$ such that
\item [(i)] For any $i_1,i_2,\ldots,i_n \in \cald,$ $P(X_1=i_1,X_2=i_2,\ldots,X_n=i_n)>0.$
\item[(ii)] For any $i_0 \in \cald$ and any sequence $(i_n)_{n \geq 1} \in \cald^\nn,$
the following limit exists:
\beq
\lim_{n \to \infty} P(X_0=i_0|X_{-k}=i_k,~1 \leq k \leq n)=P(X_0=i_0|X_{-k}=i_k,~k \geq 1),
\feq
where the right-hand side is a regular version of the conditional
probabilities.
\item[(iii)] (fading memory)
For $n \geq 0$ let
\beq
\gamma_n=\sup \left\{ \left|
\frac{P(X_0=i_0|X_{-k}=i_k,~k \geq 1)}{P(X_0=j_0|X_{-k}=j_k,~k \geq 1)} -
1\right|:i_k=j_k,~k=1,\ldots,n \right\}.
\feq
Then, the numbers $\gamma_n$ are all finite and $\limsup_n \log \gamma_n/n <0.$
\end{definition}
C-chains are a particular case of {\em chains
of infinite order} ({\em chains with complete connections}), see e.g.
\citep{kaijser,iosif-grigor}. The distributions of C-chains (a particular case of
$g$-measures introduced in \citep{keane72}) are Gibbs states
in the sense of Bowen (also known as Dobrushin-Lanford-Ruelle states),
see e.g. \citep{bowen,lalley-reg}.
\par
We have, see also \citep{berbee,fm1},
\begin{theorema*}
\citep{lalley-reg}
\label{lbrt}
Let $(X_n)_{n \in \zz}$ be a C-chain with alphabet $\cald,$
$\cals=\nolinebreak\bigcup_{n\geq 1} \cald^n,$ and $\zeta:\cals \to \cald$
be the projection into the last coordinate, i.e. $\zeta \bigl( (s_1,\ldots,s_n) \bigr)=s_n.$
Then (in an enlarged probability space) there exist a stationary irreducible Markov chain $(Y_n)_{n \in \zz}$
on $\cals,$ and constants $r \in \nn,$ $\delta>0,$  s. t.
the following Markov representation holds: $X_n=\zeta(Y_n),$ $n\in\zz.$
\item[] Furthermore, for any $(y_n)_{n\in\nn}\in\cald^\nn$ and $s\in\nn,$
\item [(i)]
$P\bigl(Y_{n+1}=(x_1,x_2,\ldots,x_t) |Y_n=(y_1,y_2,\ldots,y_s) \bigr)=0$
unless either $t=1$ or $t=s+1$ and $x_i=y_i$ for all $i\leq s,$
\item [(ii)]
$P\bigl(Y_{n+1}=(y) |Y_{n+1}\in\cald,\,Y_n=(y_1,y_2,\ldots,y_s) \bigr)=P\bigl(Y_0=(y) |Y_0\in \cald).$
\item [(iii)] $P\bigl(Y_{n+1} \in \cald |Y_n=(y_1,y_2,\ldots,y_{mr}) \bigr)=\delta,$
for all $m \in \nn.$
\item [] Moreover, without loss of generality, we can assume that
\item [(iv)]  $r$ is {\bf \em even} (see the beginning of the proof in \citep[p.~1266]{lalley-reg}).
\end{theorema*}
\noindent
The following result is proved in Section~\ref{gipr}.
\begin{theorem}
\label{inf-gk}
Assume that the coefficients $(Q_n,M_n)_{n\in\zz}$ are induced by a $C$-chain $(X_n)_{n\in\zz}.$
Furthermore, suppose that
\item [(i)]
$|Q_0| < q_0$ and $m_0^{-1}< |M_0| <m_0,$ $\as,$ for some constants $q_0>0,$ $m_0>0.$
\item [(ii)]
$\Lambda(\beta):=\limsup\limits_{n \to \infty} \fracd{1}{n} \log E
\left(\prod_{i=0}^{n-1} |M_i|^\beta \right) \geq 0$ changes its sign in $(0,\infty).$
\item [(iii)] $P(\log |M_0|=\delta \cdot k~\mbox{\rm for some}~k\in \zz|M_0 \neq 0)<1$ for all $\delta>0.$
\par
Let $\calx=(X_n)_{n\geq 0}.$ Then, with $\alpha>0$ defined below in \eqref{kappa-239}:
\item [(a)] $\lim\limits_{t \to \infty} t^\alpha P(R\cdot \eta>t|\calx)=K_\eta(\calx)$
for some bounded $K_\eta:\cald^{\zz_+} \to \rr_+,$ $\eta\in\{-1,1\}.$ Moreover,
the convergence is uniform on $\calx.$
\item [(b)] If $P(M_0<0)>0$ then $P\bigl(K_1(\calx)=K_{-1}(\calx)\bigr)=1.$
\item[(c)] If $P(Q_0>0,M_0>0)=1$ then $P\bigl(K_1(\calx)>0\bigr)=1.$
\item[(d)] For $\eta \in \{-1,1\},$ if $P\bigl(K_\eta(\calx)>0\bigr)>0$
then $P\bigl(K_\eta(\calx)>0\bigr)=1.$
\item [(e)] $P\bigl(K_1(\calx)=K_{-1}(\calx)=0\bigr)=1$
iff there is a function $\Gamma : \cald \to \rr$ s. t.
\beqn
\label{aperf}
P\bigl(Q_0+\Gamma(X_0)M_0=\Gamma(X_1)\bigr)=1.
\feqn
Moreover, if $P\bigl(Q_{0,i}=q,M_{0,i}=m\bigr)<1$ for all pairs $(q,m)\in\rr^2$ and any $i\in\cald,$
then $P\bigl(K_1(\calx)=K_{-1}(\calx)=0\bigr)=1$ iff $P\bigl(Q_0+c M_0=c\bigr)=1$ for some $c\in\rr.$
\end{theorem}
It turns out (see below, in Section~\ref{gipr}) that the $\limsup$ in (ii) is in fact a limit, and
the parameter $\alpha>0$ is (uniquely) determined by
\beqn
\label{kappa-239}
\Lambda(\alpha)=0\,,\quad\mbox{\rm where}\,\,
\Lambda(\beta)=\lim_{n \to \infty} \fracd{1}{n} \log E
\left(\proda_{i=0}^{n-1} |M_i|^\beta\right).
\feqn
We remark that condition \eqref{aperf} is a natural generalization of the criterion that
appears in the i.i.d. case, see Theorem~B in Section~\ref{gipr} below.
\par
Following the idea of \citep{mars}, Theorem~\ref{inf-gk} is obtained in Section~\ref{gipr} rather directly from
its i.i.d. prototype (cited as Theorem~B below) by applying a Doeblin's ``cyclic trick"
and associating $R$ to a linear recursion with i.i.d. coefficients. Though a reduction to the main results of \citep{omar} is possible,
it would anyway entail considerable extra work. The approach taken here enjoys the finite range and mixing properties of $X_n,$
in addition to the existence of $Y_n.$ It is much ``lighter" than those used in \citep{omar} for general Markov chains and in \citep{saporta} for a finite state case,
both built on techniques of \citep{goldie}.
\par
We conclude with the remark that,
using the Markov representation for $C$-chains, it is straightforward to deduce
the following from Theorem~\ref{negs}.
\begin{theorem}
\label{inf-g}
Assume that the coefficients $(Q_n,M_n)_{n\in\zz}$ are induced by a $C$-chain $(X_n)_{n\in\zz}.$
Furthermore, suppose that there exists $\alpha>0$ such that:
\item[(i)]
There is a slowly varying function $L(t)$ and two sequences of
constants $(q_i^{(\eta)})_{i\in\cald},$ $\eta\in\{-1,1\},$ such that for all $n\in\zz$
$\lim_{t\to\infty} \frac{P(Q_{n,i}\cdot \eta>t)}{t^{-\alpha}L(t)}=q_i^{(\eta)}.$
\item[(ii)] There exists $\beta>\alpha$ such that $\sup_{i\in\cald} E(|M_{0,i}|^\beta)<\infty.$
\par
Let $\calx=(X_n)_{n\geq 0}.$ Then, for some bounded function $K^{(\eta)},$ $\eta\in\{-1,1\},$
we have $P(R\cdot \eta >t|\calx)~\sim K^{(\eta)}(\calx) t^{-\alpha}L(t),$ $\as$
Moreover, $P\bigl(K^{(1)}(\calx)\not =0\bigr)>0.$
\end{theorem}
\section{Proof of Theorem~\ref{vart}}
\label{proofvart}
The key to the result is Proposition~\ref{klemma} extending the corresponding statement in \citep{trakai75,grey}.
\begin{proposition}
\label{klemma} Let $Y$ be a random variable such that:
\item[(i)] $Y \in \sigma(X_n, Q_{n,i},M_{n,i}:n \leq 0,i \in \cald\bigr).$
\item[(ii)] For $\eta\in\{-1,1\}$ there exist non-negative constants $\bigl(c_i^{(\eta)}\bigr)_{i \in \cald}$ such that
\begin{itemize}
\item [(a)] $\lim_{t\to\infty} \frac{P(Y\cdot \eta >t|X_0=i)}{t^{-\alpha} L(t)}=c_i^{(\eta)},$ uniform on $i\in\cald.$
\item [(b)] $\sup_{i\in\cald} c_i^{(\eta)}<\infty.$
\end{itemize}
Then $\lim\limits_{t\to\infty} \frac{P(Q_1+M_1Y >t|X_1=i)}{t^{-\alpha}L(t)}=q_i^{(1)}+m_i^{(1)}\sum\limits_{j\in\cald}H(i,j)c_j^{(1)}$
uniformly on $i\in\cald.$
\end{proposition}
\begin{proof}
Let $(Y_i)_{i\in \cald}$ be random variables independent of both $(X_n)_{n \in \zz}$ and
$(Q_{1,i},M_{1,i})_{i \in \cald},$ such that $P(Y_i\leq
t)=P(Y\leq t|X_0=i)$ for all $t\in \rr$ and $i\in\cald.$ It follows from $P\bigl((Q_1,M_1,Y)\in \cdot|X_1=i,X_0=j\bigr)=
P\bigl((Q_{1,i},M_{1,i},Y_j)\in\cdot)$ that
\beqn
\nonumber
&&
P(Q_1+M_1Y >t|X_1=i)=
\\
\nonumber
&&
\qquad
\qquad
\qquad
\qquad
=\sum_{j \in \cald} P(Q_1+M_1Y>t|X_1=i,X_0=j)H(i,j)
\\
\label{unic}
&&
\qquad
\qquad
\qquad
\qquad
=
\sum_{j \in \cald} P(Q_{1,i}+M_{1,i} Y_j>t)H(i,j).
\feqn
By \citep[Lemma~2]{grey}, which is the i.i.d. prototype of our proposition,
\beqn
\label{uniform}
P(Q_{1,i}+M_{1,i} Y_j>t)\sim t^{-\alpha}L(t)\bigl(q_i^{(1)}+c_j^{(1)}m_i^{(1)}\bigr).
\feqn
To complete the proof, it suffices to show that the convergence in \eqref{uniform}
is uniform on $i,j.$ To this end we decompose $P(Q_{1,i}+M_{1,i} Y_j>t)$ into individually tractable terms,
as in \citep[Lemma~2]{grey}. Fix $\veps\in (0,1)$ and write
\beq
P(Q_{1,i}+M_{1,i} Y_j>t)=A_{i,j}^{(1)}(t)-A_{i,j}^{(2)}(t)+A_{i,j}^{(3)}(t)+A_{i,j}^{(4)}(t),~\mbox{where}
\feq
$A_{i,j}^{(1)}(t)= P\bigl(Q_{1,i}>t(1+\veps)\bigr),$
$A_{i,j}^{(2)}(t)= P\bigl(Q_{1,i}>(1+\veps)t,\, Q_{1,i}+M_{1,i}Y_j\leq t\bigr),$\\
$A_{i,j}^{(3)}(t)= P\bigl(|Q_{1,i}-t|\leq \veps t,\, Q_{1,i}+M_{1,i}Y_j>t\bigr),$ \\
$A_{i,j}^{(4)}(t)= P\bigl(Q_{1,i}\leq (1-\veps)t,\,Q_{1,i}+M_{1,i}Y_j>t\bigr).$\\
By (A1), $\frac{A_{i,j}^{(1)}(t)}{t^{-\alpha}L(t)}$ converges uniformly in $i,j$ to $q_i^{(1)}(1+\veps)^{-\alpha}$.
For $A_{i,j}^{(2)}$ write
\beq
&&
A_{i,j}^{(2)}(t)=P\bigl(Q_{1,i}>(1+\veps)t,\,
Q_{1,i}+M_{1,i}Y_j\leq t;\, Y_j<-\veps t^{\frac{\beta - \alpha}{2\beta}}\bigr)
\\
&&
\,\,\,
+ P\bigl(Q_{1,i}>(1+\veps)t,\, Q_{1,i}+M_{1,i}Y_j\leq t;\, Y_j\geq -\veps t^{\frac{\beta - \alpha}{2\beta}}\bigr)
\\
&&
\,\,\,
\leq  P\bigl(M_{1,i}Y_j\leq -\veps t,\, Y_j\geq -\veps t^{\frac{\beta - \alpha}{2\beta}}\bigr)+
P\bigl(Q_{1,i}>(1+\veps)t,\, Y_j< -\veps t^{\frac{\beta - \alpha}{2\beta}}\bigr)
\\
&&
\,\,\,
\leq P\bigl(M_{1,i}\geq t^{\frac{\alpha + \beta}{2\beta}}\bigr)+
P\bigl(Q_{1,i}>(1+\veps)t\bigr)P\bigl(Y_j< -\veps t^{\frac{\beta - \alpha}{2\beta}}\bigr).
\feq
To obtain a bound on $\frac{A_{i,j}^{(2)}(t)}{t^{-\alpha}L(t)}$ tending to zero uniformly on $i,j\in\cald$ as $t\to\infty,$
we use (A2), Chebyshev's inequality $P\bigl(M_{1,i}\geq t^{\frac{\alpha + \beta}{2\beta}}\bigr) \leq
E\bigl(M_{1,i}^\beta\bigr) \cdot t^{-\frac{\alpha+\beta }{2}},$ and the inequality $P\bigl(Y_{j}<-\veps t^{\frac{\beta - \alpha}{2\beta}}\bigr)\leq C
\bigl(\veps t^{\frac{\beta - \alpha}{2\beta}}\bigr)^{-\alpha} L\bigl(\veps t^{\frac{\beta - \alpha}{2\beta}}\bigr),$ which is true
for some $C>0$ in virtue of condition (ii) of the proposition. A uniform bound on
$\frac{A_{i,j}^{(3)}(t)}{t^{-\alpha}L(t)}$ which tends to $0$ as $\veps\to 0$ follows directly from (A1).
Finally, denote $g_{j,t}(a,b)=P\bigl(Y_j>a^{-1}(t-b)\bigr),$ fix constants $m\geq 1$ and $n\geq 1,$ and let $A_{i,j}^{(4)}(t)=A_{i,j}^{(4,1)}(t)+A_{i,j}^{(4,2)}(t)+A_{i,j}^{(4,3)}(t),$
where
\beq
A_{i,j}^{(4,1)} (t)&=& E\bigl( g_{j,t}(M_{1,i},Q_{1,i})
\odin{Q_{1,i}\leq(1-\veps)t}\odin{M_{1,i}>m} \bigr)\\
A_{i,j}^{(4,2)} (t)&=& E\bigl( g_{j,t}(M_{1,i},Q_{1,i})
\odin{Q_{1,i}\leq(1-\veps)t}\odin{M_{1,i}\leq m}\odin{|Q_{1,i}|>n} \bigr)
\\
A_{i,j}^{(4,3)} (t)&=& E\bigl( g_{j,t}(M_{1,i},Q_{1,i})
\odin{Q_{1,i}\leq(1-\veps)t}\odin{M_{1,i}\leq m}\odin{|Q_{1,i}|\leq n} \bigr).
\feq
Note that $\frac{A_{i,j}^{(4,1)} (t)}{t^{-\alpha}L(t)}\leq
E\Bigl(\frac{g_{j,\veps t}(M_{1,i},0)}{(M_{1,i}^{-1}\veps t)^{-\alpha}
L(M_{1,i}^{-1}\veps t)}\frac{(M_{1,i}^{-1}\veps t)^{-\alpha}L(M_{1,i}^{-1}\veps t)}{t^{-\alpha}L(t)}\odin{M_{1,i}>m}\Bigr).$
Hence, by condition (ii)-(b) of the proposition,
$\frac{A_{i,j}^{(4,1)} (t)}{t^{-\alpha}L(t)}\leq E\bigl(\frac{C M_{1,i}^\alpha}{\veps^\alpha}\cdot \frac{L(M_{1,i}^{-1}\veps t)}{L(t)}\odin{M_i>m }\bigr)$ for some constant $C>0.$ By \citep[Lemma~1]{grey}, for any $\delta>0$ there is $K=K(\delta)>0$ such that
$\sup_{t>0}\frac{L(\lambda t)}{L(t)}\leq \max\{\lambda^\alpha, K\lambda^{-\delta}\}$ for all $\lambda>0.$
Using $\lambda=M_{1,i}^{-1}\veps$ and $\delta=\frac{\beta-\alpha}{2},$ we obtain
\beq
\frac{A_{i,j}^{(4,1)} (t)}{t^{-\alpha}L(t)}&\leq& E\bigl(CM_{1,i}^\alpha \veps^{-\alpha}\cdot \max\{(M_{1,i}^{-1}\veps)^\alpha,
K(M_{1,i}^{-1}\veps)^{-\frac{\beta-\alpha}{2}}\}\odin{M_{1,i}>m}\bigr) \\
&\leq& E\bigl(CM_{1,i}^\alpha \veps^{-\alpha}(M_{1,i}^{-\alpha}\veps^\alpha+KM_{1,i}^{\frac{\beta-\alpha}{2}}
\veps^{-\frac{\beta-\alpha}{2}})\odin{M_{1,i}>m}\bigr).
\feq
H\"{o}lder's inequality with $p=\frac{2\beta}{\alpha+\beta},$
$q=\frac{2\beta}{\beta-\alpha}$ yields
\beq
\frac{A_{i,j}^{(4,1)} (t)}{t^{-\alpha}L(t)}&\leq& C \veps^{-\frac{\alpha+\beta}{2}}
E\bigl[\bigl(1+KM_{1,i}^{\frac{\beta+\alpha}{2}} \bigr)\odin{M_{1,i}>m}\bigr]
\\
&\leq&
C \veps^{-\frac{\alpha+\beta}{2}}
E\bigl[\bigl(1+KM_{1,i}^{\frac{\beta+\alpha}{2}}\bigr)^{\frac{2\beta}{\alpha+\beta}}\bigr] ^{\frac{\alpha+\beta}{2\beta}}
P(M_{1,i}>m)^{\frac{\beta-\alpha}{2\beta}}
\\
&\leq&
C\veps^{-\frac{\alpha+\beta}{2}}
E\bigl[\bigl(1+K^{\frac{2\beta}{\beta+\alpha}}M_{1,i}^\beta\bigr)\cdot
2^{\frac{\beta-\alpha}{\alpha+\beta}}\bigr]^{\frac{\alpha+\beta}{2\beta}}
P(M_{1,i}>m)^{\frac{\beta-\alpha}{2\beta}},
\feq
where we used the inequality $(x+y)^p \leq 2^{p-1}(x^p+y^p)$ which is valid for $x,y\geq 0$ and $p>1.$
Since $P(M_{1,i}>m)\leq m^{-\beta} E\bigl(M_{1,i}^\beta\bigr),$ it follows from (A2)
that $\frac{A_{i,j}^{(4,1)}(t)}{t^{-\alpha}L(t)}$ is uniformly bounded by a
function of $m$ which tends to zero when $m\to\infty.$ The term $A_{i,j}^{(4,2)}(t)$
is treated similarly, and we therefore omit the details.
We next show the asymptotic of $A_{i,j}^{(4,3)}(t).$ If $Q_{1,i}\leq(1-\veps)t$ and $M_{1,i}\leq m,$ then
$M_{1,i}^{-1}(t-Q_{1,i})\geq m^{-1}\veps t.$ Hence, in virtue of (A1) and condition (ii) of the proposition,
we have $P-\as,$ uniformly on $i,j\in \cald,$
\beqn
\label{first}
\frac{g_{j,t}(M_{1,i},Q_{1,i})\odin{Q_{1,i}\leq(1-\veps)t}\odin{M_{1,i}\leq m}}
{\bigl(M_{1,i}^{-1}(t-Q_{1,i})\bigr)^{-\alpha}
L\bigl(M_{1,i}^{-1}(t-Q_{1,i})\bigr)} \to_{t\to\infty}
 c_j^{(1)} \odin{M_{1,i}\leq m}.
\feqn
Furthermore, with probability one, uniformly on $i\in\cald,$
\beqn
\label{second}
t^\alpha \bigl(M_{1,i}^{-1}(t-Q_{1,i})\bigr)^{-\alpha}
\odin{M_{1,i}\leq m,|Q_{1,i}|\leq n}\to_{t\to\infty}
M_{1,i}^\alpha \odin{M_{1,i}\leq m, |Q_{1,i}|\leq n}.
\feqn
By  \citep[Theorem~1.2.1]{rvariation}, we have $\as,$ uniformly on $i\in\cald,$
\beqn
\label{third}
\frac{L\bigl(M_{1,i}^{-1}(t-Q_{1,i})\bigr)}{L(t)}\odin{m^{-1}<M_{1,i}\leq m, |Q_{1,i}|\leq n}  \to_{t\to\infty}
\odin{m^{-1}<M_{1,i}\leq m, |Q_{1,i}|\leq n}.
\feqn
By \citep[Theorem~1.5.6]{rvariation}, $\forall~\delta>0~\exists~t_0=t_0(\delta)$ such that
\beqn
\label{forth}
\frac{L\bigl(M_{1,i}^{-1}(t-Q_{1,i})\bigr)}{L(t)}\odin{M_{1,i}\leq m^{-1},|Q_{1,i}|\leq n}
\leq \frac{mt}{t-n}\cdot \odin{M_{1,i}\leq m^{-1}},\quad t>t_0.
\feqn
Estimates \eqref{first}-\eqref{forth} along with assumption (A4) and the bounded convergence theorem show
that $\lim\limits_{t\to\infty} \frac{A_{i,j}^{(4)}(t)}{t^{-\alpha}L(t)}=c_j^{(1)}E\bigl(M_{1,i}^\alpha\bigr)$ uniformly on $i,j.$
\end{proof}
To enable us to use Proposition~\ref{klemma} iteratively we need the following:
\begin{lemma}
\label{acon}
Let $Y$ satisfy the conditions of Proposition~\ref{klemma},
and let $\witi Y=Q_1+M_1Y.$ Then $\witi Y$ satisfies condition (ii) of the proposition.
\end{lemma}
\begin{proof}
Apply Proposition~\ref{klemma} to $(Y,Q_1,M_1)$ and $(-Y,-Q_1,M_1).$
\end{proof}
The next technical lemma is immediate from Proposition~\ref{klemma} and (A3).
\begin{lemma}
\label{domin} $\exists$ a
random variable $Z\geq 0$ satisfying the conditions of Proposition~\ref{klemma},
s. t. $P(Q_1+M_1Z >t|X_1=i) \leq P(Z>t|X_0=i),$ $t>0,$ $i \in \cald.$
\end{lemma}
We are now in position to complete the proof of Theorem~\ref{vart}.
\begin{lemma}
\label{lisup}
For all $i\in\cald,$ $\limsup\limits_{t\to\infty}
\fracd{P(R>t|X_0=i)}{t^{-\alpha}L(t)} \leq (I-G)^{-1}\bq^{(1)}(i).$
\end{lemma}
\begin{proof}
Let $R_0=Z,$ where $Z$ is as in Lemma~\ref{domin}. Then, for all $t>0$ and $i \in \cald,$
we have $P(R_1>t|X_1=i)\leq P(R_0>t|X_0=i).$  This yields:
\beq
&&
P(R_2>t|X_2=i)= \sum_{j\in\cald}P(Q_2+M_2R_1>t|X_2=i,X_1=j)H(i,j)\\
&&
\quad
= \sum_{j\in\cald}P(Q_{2,i}+M_{2,i}R_1>t|X_1=j)H(i,j) \\
&&
\quad
\leq \sum_{j\in\cald}P(Q_{1,i}+M_{1,i}R_0|X_0=j)H(i,j) =P(R_1>t| X_1=i).
\feq
Therefore $P(R_2>t|X_2=i)\leq P(R_1>t|X_1=i).$ Iterating,
we obtain
\beqn
\label{iter}
P(R_n>t|X_n=i)\leq P(R_{n-1}>t|X_{n-1}=i),\quad \forall~n\in \nn,\,i\in\cald,\,t>0.
\feqn
By Proposition~\ref{klemma},
$\frac{P(R_n>t|X_n=i)}{t^{-\alpha}L(t)}\sim\bigl[\bq^{(1)}+\cdots+G^{n-1} \bq^{(1)} + c^* G^n \one\bigr](i),$
where $c^*>0$ is a constant such that $P(Z>t)\sim c^*L(t)t^{-\alpha}$ and $\one\in \rr^\cald$
has all components equal to $1.$ Since $P(Z>0)=1,$ it follows from \eqref{iter} that $P(R_n>t|X_n=i) \geq P(R>t|X_0=i)$ for $n\geq 0.$
Hence,  $\limsup_{t \to \infty}
\fracd{P(R>t|X_0=i)}{t^{-\alpha}L(t)} \leq (I-G)^{-1}\bq^{(1)}(i)$ for all $i\in\cald.$
\end{proof}
\begin{lemma}
\label{linf}
For all $i\in\cald,$ $\liminf_{t\to\infty}
\fracd{P(R>t|X_0=i)}{t^{-\alpha}L(t)} \geq (I-G)^{-1}\bq^{(1)}(i).$
\end{lemma}
\begin{proof}
Let $\calr=Q_{-1}+M_{-1}Q_{-2}+M_{-1}M_{-2}Q_{-3}+\ldots.$
Let $R_0\geq 0$ be independent of
$(X_n,Q_{i,n},M_{i,n})_{n\geq 1,i\in\cald},$ s. t.
$P(R_0>t)=P(\calr>0,Q_0>t)$ for $t>0.$ Then $P(R_0>t|X_0=i)\leq P(Q_0+M_0\calr>t|X_0=i)=$
$P(R>t|X_0=i).$ We will now use induction to show that for $n\geq 0,$
\beqn
\label{monot}
P(R_n>t|X_n=i)\leq P(R>t|X_0=i)\qquad \forall~t>0,\,i\in\cald.
\feqn
Specifically, assuming \eqref{monot} for some $n \geq 0$ we obtain
\beq
&&
P(R_{n+1}>t | X_{n+1}=i)=\sum_{j\in\cald}P(Q_{n+1,i}+M_{n=1,i} R_n>t |X_n=j) H(i,j)
\\
&& \qquad
\leq \sum_{j\in\cald}P(Q_{1,i}+M_{1,i}R>t|X_0=j) H(i,j)
\\
&& \qquad
= P(Q_1+M_1R>t|X_1=i )=P(R>t | X_0=i),
\feq
Moroever, uniformly on $i\in\cald,$ $\frac{P(R_0>t|X_0=i)}{t^{-\alpha}L(t)}
= \frac{P(Q_{0,i}>t)P(\calr>0|X_0=i)}{t^{-\alpha}L(t)}\sim e_i,$
where $e_i=q_i^{(1)} P(\calr>0|X_0=i).$ Let $\be=(e_i)_{i\in\cald}.$ Then, by Proposition~\ref{klemma},
$P(R_n>t|X_n=i)\sim t^{-\alpha}L(t)\cdot \bigl[\bq^{(1)}+G\bq^{(1)}+\ldots+G^{n-1} \bq^{(1)} + G^n \be\bigr](i).$
This completes the proof of Lemma~\ref{linf} in view of \eqref{monot}.
\end{proof}
\section{Proof of Theorem~\ref{negs}}
\label{proofa}
The following result extends Lemma~4 of \citep{grey}.
\begin{lemma}
\label{knegs-lemma}
Let $Y \in \sigma\bigl(X_n,Q_{n,i},M_{n,i}:n \leq 0,i \in \cald\bigr)$
be a random variable s. t. $c_i^{(\eta)}:=\limsup\limits_{t\to\infty} \frac{P(Y\cdot \eta >t|X_0=i)}{t^{-\alpha} L(t)}$ and
$d_i^{(\eta)}:=\liminf\limits_{t\to\infty} \frac{P(Y\cdot \eta>t|X_0=i)}{t^{-\alpha} L(t)}$ are finite
for all $i\in\cald$ and $\eta\in\{-1,1\}.$
Then for all $i \in \cald,\eta\in\{-1,1\},$ \\
$\limsup\limits_{t\to\infty}\frac{P\bigl((Q_1+M_1Y)\cdot \eta >t\bigl|X_1=i\bigr)}{t^{-\alpha}L(t)}
\leq q_i^{(\eta)}+\sum_{j\in\cald}\sum_{\gamma\in\{-1,1\}}G_{\gamma}(i,j)c_j^{(\gamma)}$ and \\
$\liminf\limits_{t\to\infty}\frac{P\bigl((Q_1+M_1Y)\cdot \eta >t\bigl|X_1=i\bigr)}{t^{-\alpha}L(t)}
\geq q_i^{(\eta)}+\sum_{j\in\cald}\sum_{\gamma\in\{-1,1\}}G_{\gamma}(i,j)d_j^{(\gamma)}.$
\end{lemma}
\begin{proof}
Let $(Y_j)_{j\in \cald}$ be random variables independent of both $(X_n)_{n \in \zz}$ and $(Q_{1,i},M_{1,i})_{i \in \cald},$
such that $P(Y_j\cdot \eta>t)=P(Y\cdot \eta> t|X_0=j)$ for $\eta\in\{-1,1\}.$
Then we have $P\bigl((Q_1+M_1Y)\cdot \eta >t|X_1=i)
=\sum_{j \in \cald} P\bigl((Q_{1,i}+M_{1,i} Y_j)\cdot \eta >t\bigr)H(i,j)$ according ro \eqref{unic}.
To complete the proof, apply \citep[Lemma~4]{grey} separately to each term $P\bigl((Q_{1,i}+M_{1,i} Y_j)\cdot \eta >t\bigr).$
\end{proof}
Let $R^*=|Q_0|+\sum_{n=1}^{\infty} |Q_{-n}| \prod_{i=0}^{n-1} |M_{-i}|$ be a stationary solution of the
equation $R_{n+1}=|Q_n|+|M_n|R_n.$ Notice that $-R^*$ is a stationary solution of the equation $R_{n+1}=-|Q_n|+|M_n|R_n.$
Since $P(-R^*\leq R\leq R^*)=1,$ Theorem~\ref{vart} ensures that Lemma~\ref{knegs-lemma} can be applied with $Y=R.$
Let $a_i^{(\eta)}=\limsup_{t\to\infty}\frac{P(R\cdot \eta>t|X_0=i)}{t^{-\alpha}L(t)}$ and
$b_i^{(\eta)}=\liminf_{t\to\infty}\frac{P(R\cdot \eta>t|X_0=i)}{t^{-\alpha}L(t)}.$ Denote
$\ba^{(\eta)}=\bigl(a_i^{(\eta)}\bigr)_{i\in\cald},$ $\bB^{(\eta)}=\bigl(b_i^{(\eta)}\bigr)_{i\in\cald}$ and
$\ba=\ba^{(-1)}+\ba^{(1)},$ $\bB=\bB^{(-1)}+\bB^{(1)}.$
The application of Lemma~\ref{knegs-lemma} yields $\ba \leq \bq+G\ba$ and $\bB \geq \bq+G\bB,$ which implies
$(I-G)^{-1}\bq\leq  \bB\leq  \ba\leq (I-G)^{-1}\bq.$ This is only possible if $\ba^{(\eta)}=\bB^{(\eta)},$
the inequalities in the conclusions of Lemma~\ref{knegs-lemma} are actually equalities, and thus
$\ba^{(\eta)}=\bq^{(\eta)}+G_1\ba^{(\eta)}+G_{-1}\ba^{(-\eta)},$
which implies the result of Theorem~\ref{negs}.
\section{Proof of Theorem~\ref{inf-gk}}
\label{gipr}
Consider the {\bf backward} chain $Z_n=Y_{-n},$ $n\in\zz,$ with transition matrix
$H(x,y)=P(Y_n=y|Y_{n+1}=x)=\frac{P(Y_0=y)}{P(Y_0=x)}\cdot P(Y_{n+1}=y|Y_n=x).$
Fix $y^* \in \cals$ and a constant $r \in (0,1).$
Let $(\eta_n)_{n \in \zz}$ be a sequence of i.i.d. ``coins"
independent of both $(Z_n)_{n \in \zz}$ and $\bigl(Q_{n,i}, M_{n,i}\bigr)_{n \in \zz,i\in\cald},$
such that $P(\eta_0=1)=r,$ $P(\eta_0=0)=1-r.$ Let
$N_0=0$ and $N_i=\inf\{n>N_{i-1}:Z_n=y^*,\eta_n=1\},$ $i \in \nn.$
The blocks $(Z_{_{N_i}},\ldots,Z_{_{N_{i+1}-1}})$ are
independent for $i\geq 0$ and identically distributed for $i\geq 1.$
Between two successive regeneration times $N_i$ the chain $(Z_n)_{n\geq 0}$ evolves according to a sub-Markov kernel $\Theta$ given by
$H(x,y)=\Theta(x,y)+r \odin{y=y^*} H(x,y),$ i.e.
\beqn
\label{between}
\Theta(x,y)=P(Z_1=y,N_1>1|Z_0=x).
\feqn
Theorem~A implies that $E\bigl(e^{\beta N_1}\bigl|Z_0\bigr)$
is uniformly bounded for some $\beta>0$ (see the paragraph following Theorem~1 in \citep{lalley-reg}).
For $i \geq 0,$ let
\beq
A_i&=&Q_{_{N_i}}+Q_{_{N_i+1}}M_{_{N_i}}+...+
Q_{_{N_{i+1}-1}} M_{_{N_i}}M_{_{N_i+1}} \ldots
M_{_{N_{i+1}-2}}
\\
B_i&=& M_{_{N_i}}M_{_{N_i+1}} \ldots M_{_{N_{i+1}-1}}.
\feq
The pairs $(A_i,B_i)$ are independent for $i\geq 0,$ identically distributed for $i\geq 1$ and
$R=A_0+\sum_{n=1}^\infty A_n \prod_{i=0}^{n-1} B_i.$
To prove Theorem~\ref{inf-gk} we will verify the conditions of
the following theorem for $(A_i,B_i)_{i \geq 1}.$
\begin{theoremb*} \citep{kesten-randeq, goldie}
Let $(A_i,B_i)_{i \geq 1}$ be i.i.d. and
\item[(i)] For some $\alpha>0,$ $E\left(|A_1|^\alpha\right)=1$ and $E\left(|B_1|^\alpha \log^+
|B_1|\right)<\infty.$
\item[(ii)] $P(\log |B_1|=\delta \cdot k~\mbox{\rm for some}~k\in\zz|B_1 \neq 0)<1$ for all $\delta>0.$
\par
Let $\witi R=A_1+\sum_{n=2}^\infty A_n \prod_{i=1}^{n-1} B_i.$ Then
\item [(a)]
$\lim\limits_{t \to \infty} t^\alpha P\bigl(\witi R>t\bigr)=K_+,$
$\lim\limits_{t \to \infty} t^\alpha P\bigl(\witi R < -t\bigr)=K_-$ for some $K_+, K_-\geq 0.$
\item [(b)]
If $P(B_1<0)>0,$ then $K_+=K_-.$
\item [(c)]
$K_++K_->0$ if and only if $P\bigl(A_1=(1-B_1)c\bigr)<1$ for all $c \in \rr.$
\end{theoremb*}
\noindent
For $\beta\geq 0$ define matrices $H_\beta$ and $\Theta_\beta$ by setting
$H_\beta(x,y)=H(x,y)E(|M_{0,\zeta(y)}|^\beta )$ and $\Theta_\beta(x,y)=\Theta(x,y)E( |M_{0,\zeta(y)}|^\beta ).$ Since
we have $E\bigl({\prod}_{i=0}^n |M_{-i}|^\beta \odin{Z_n=y}\Bigl| Z_0=x\bigr)=
E\bigl(|M_0|^\beta\bigl|Z_0=x\bigr) H_\beta^n(x,y)$ and, for $y\not =y^*,$
\beq
E\Bigl(\odin{n<N_1,Z_n=y} {\prod}_{i=0}^n |M_{-i}|^\beta \bigl| Z_0=x\Bigr)=
E\bigl(|M_0|^\beta\bigl|Z_0=x\bigr) \Theta_\beta^n(x,y),
\feq
\citep[Lemma~2.3 and Proposition~2.4]{omar} show that \eqref{kappa-239} holds and that the spectral radius of $H_\alpha$ is 1,
while the spectral radius of $\Theta_\alpha$ is less than 1. Hence, see (2.43) in \citep{mars}, we have:
\begin{lemma}
\label{kapm}
$E(|B_1|^\alpha)=1.$
\end{lemma}
We next show that $\log |B_0|$ is a non-lattice random variable if $y^*\in\cald.$
\begin{lemma}
\label{lef1}
Assume $y^*\in\cald^1$ and let $ V=\sum\limits_{n=0}^{N_1-1} \log |M_{-n}|=\log |B_0|.$ Then
for any $\delta>0$ we have $P\bigl(V \in \delta\cdot\zz\bigl|Z_0=y^*\bigr)<1,$ where $\delta \cdot \zz:=\{k\cdot \delta:k\in\zz\}.$
\end{lemma}
\begin{proof}
For $y\in\cald$ and $k\in\nn,$ let $y_{(k)}$ denote $(y_1,\ldots,y_k)\in\cald^k$ such that
$y_2=\cdots=y_k=y$ and $y_1=y^*.$ By Lalley's Theorem~A, for any $y\in \cald,$
\beqn
\label{impl}
P(Z_0=y_{(1)},Z_1=y_{(m)},\ldots,Z_{m-1}=Z_{N_1-1}=y_{(2)})>0
\feqn
for some $m>1.$ Therefore, $P\bigl(V \in \delta\cdot\zz|Z_0=(y^*)\bigr)=1$ along with
\eqref{impl} imply: (i) using $y=y^*,$ $P\bigl(m\cdot \log |M_{0,y^*}|
\in \delta \cdot \zz|Z_0=(y^*)\bigr)=1$; (ii) using general $y\in\cald,$
$P\bigl((m-1)\cdot \log |M_{0,y}|+\log |M_{0,y^*}|
\in \delta \cdot \zz\bigl|Z_0=(y^*)\bigr)=1.$ Therefore, we would have $P\bigl(m(m-1)\cdot \log |M_{0,y}|
\in \delta \cdot \zz\bigl|Z_0=(y^*)\bigr)=1$ for any $y\in \cald,$ contradicting condition (iii)
of Theorem~\ref{inf-gk}.
\end{proof}
The proof of the next lemma is verbatim the proof of (2.44) in \citep{mars}, and therefore is omitted.
\begin{lemma}
\label{lemf}
$\exists~\beta>\alpha$ s. t. $E\bigl[\bigl(\sum_{n=0}^{N_1-1} \prod_{i=0}^{n-1} |M_{-i}| \bigr)^\beta\bigl| Z_0\bigr]$
is bounded $\as$
\end{lemma}
\paragraph{Completion of the proof of Theorem~\ref{inf-gk}} $\mbox{}$ \\
{\bf (a)--(d)} It follows from Lemmas~\ref{kapm}--\ref{lemf} that the conclusions of
Theorem~B can be applied to $(A_1,B_1).$
Claims (a) through (d) of Theorem~\ref{inf-gk} follow then from the identity $R=A_0+B_0\witi R$ and the
independence of $(A_0,B_0)$ of $\witi R$ under the conditional measures $P(\cdot|Z_0=x).$ In particular,
for $\eta\in\{-1,1\},$
\beqn
\label{kfor}
\lim_{t\to\infty}P\bigl(R\cdot \eta >t|Z_0)=E\bigl(|B_0|^\alpha(\odin{B_0\cdot \eta >0} K_++ \odin{B_0 \cdot \eta <0} K_-)\bigl|Z_0\bigr)
\feqn
Note that Theorem~A-(iv) implies $P(N_1~\mbox{is odd})>0$ and
$P(N_1~\mbox{is even})>0.$ Therefore, $P(M_0<0)>0$ yields $P(B_0<0|Z_0=x)>0$ for all $x\in \cals.$
\\
{\bf (e)} First, \eqref{aperf} implies $P\bigl(R_n=\Gamma(X_n)\bigr)=1$ provided that $R_0=\Gamma(X_0).$
Hence if \eqref{aperf} is true and $R_0=\Gamma(X_0),$ then $R_n$ can take only a finite number of values from $\{\Gamma(i):i\in\cald\}$
and cannot converge in distribution to a power-tailed random variable.
Therefore, \eqref{aperf} implies $P\bigl(K_1(\calx)+K_{-1}(\calx)=0\bigr)=1.$
\par
Assume now $P\bigl(K_1(\calx)+K_{-1}(\calx)=0\bigr)=1.$ Then
\eqref{kfor} and Theorem~\ref{inf-gk}-(b) imply $K_+=K_-=0.$ Then, by Theorem~B,
$P\bigl(A_1+c(y^*)B_1=c(y^*)\bigr)=1$ for a constant $c(y^*).$ Hence
$P\bigl(Q_0+M_0 \widehat R=c(y^*)\bigl|Z_0=y^*\bigr)=1,$
with $\widehat R=\frac{R-Q_0}{M_0}.$ Since $\widehat R$ is
conditionally independent of $(Q_0,M_0)$ given $Z_0,$
\beqn
\label{fina}
&&P\bigl(\widehat R=c_1(y^*)\bigl|Z_0=y^*\bigr)=P\bigl(R=c_1(y^*)\bigl|Z_{-1}=y^*\bigr)=1,
\\
\label{fina4}
&&P\bigl( Q_{0,\zeta(y^*)}+ M_{0,\zeta(y^*)} c_1(y^*)=c(y^*)\bigr)=1,
\feqn
for some constant $c_1(y^*).$ It follows from \eqref{fina} and the Markov property that $P\bigl(R=c_1(y^*)\bigl|Z_0=y\bigr)=1$
$\forall~y\in\cals$ s. t. $P(Z_0=y|Z_{-1}=y^*)>0.$ Then
\beqn
\label{fina1}
P\bigl(R=c(y^*)\bigl|Z_{-1}=x\bigr)=1~\forall~x\in\cals~\mbox{s. t.}~
P(Z_0=y^*|Z_{-1}=x)>0.
\feqn
Thus $P\bigl(R=c_1(x)=c(y^*)\bigl|Z_{-1}=x\bigr)=1$ whence $P(Z_0=y^*|Z_{-1}=x)>0.$ Since $y^*\in\cals$ is arbitrary,
$P\bigl(Q_{0,\zeta(Z_0)}+ M_{0,\zeta(Z_0)} c_1(Z_0)=c_1(Z_{-1})\bigr)=1$
due to \eqref{fina4}. This is exactly \eqref{aperf} once one sets $\Gamma(Y_n)=c_1(Z_{-n}).$
\par
Suppose now in addition that $(Q_{0,i},M_{0,i}),$ $i\in\cald,$ are non-degenerate.
Since the system $q+mc_1=c,$ $q+m\tilde c_1=\tilde c$ for unknown variables $q,m$
has a unique solution unless $c_1=\tilde c_1,$ it follows from
\eqref{fina4} that $c_1$ depends on $\zeta(y^*)$ only. Let $\Gamma(i)=c_1(y^*)$ for $i=\zeta(y^*).$
Then, \eqref{fina4} and \eqref{fina1} imply $P\bigl(Q_0+M_0\Gamma(X_0)=\Gamma(X_1)\bigr)=1.$
Hence, in virtue of (i) of Definition~\ref{typc}, $P(\Gamma(X_n)=c)=P\bigl(Q_0+M_0c=c\bigr)=1$
for some constant $c\in \rr.$

\end{document}